\documentclass[11pt]{amsart}
\usepackage{amsmath,amssymb,enumerate}

\newtheorem{thm}{Theorem}[section]
 \numberwithin{equation}{section} 
 \numberwithin{figure}{section} 
 \theoremstyle{plain}
 \theoremstyle{plain}    
 \newtheorem{cor}[thm]{Corollary} 
 \theoremstyle{plain}    
 \newtheorem{defi}[thm]{Definition}
 \theoremstyle{plain}    
 \theoremstyle{remark}
 \newtheorem{rem}[thm]{Remark}
 \theoremstyle{definition}

\newcommand{\R}{\mathbb R}

\newcommand{\C}{\mathbb C}

\renewcommand{\a}{\alpha}

\newcommand{\e}{\varepsilon}
\newcommand{\f}{\varphi}

\newcommand{\p}{\psi}

\begin{document}

\title[Continuous approximation of qpsh functions]{Continuous approximation of quasiplurisubharmonic  functions}

\author{Philippe Eyssidieux, Vincent Guedj, Ahmed Zeriahi}

\address{Universit\'e Joseph Fourier et Institut Universitaire de France, France}

\email{Philippe.Eyssidieux@ujf-grenoble.fr }

\address{Institut de Math\'ematiques de Toulouse et Institut Universitaire de France,   \\ Universit\'e Paul Sabatier \\
118 route de Narbonne \\
F-31062 Toulouse cedex 09\\}

\email{vincent.guedj@math.univ-toulouse.fr}

\address{Institut de Math\'ematiques de Toulouse,   \\ Universit\'e Paul Sabatier \\
118 route de Narbonne \\
F-31062 Toulouse cedex 09\\}

\email{zeriahi@math.univ-toulouse.fr}

\begin{abstract} 
Let $X$ be a compact K\"ahler manifold and $\theta$ a smooth closed $(1,1)$-real form representing a
big cohomology class $\a \in H^{1,1}(X,\R)$. 
The purpose of this note is to show, using pluripotential and viscosity techniques, that any $\theta$-plurisubharmonic function
$\f$ can be approximated from above by a decreasing sequence of {\it continuous}
$\theta$-plurisubharmonic functions with minimal singularities, assuming that there exists a single such function.
 \end{abstract}

\date{\today \\
 The authors are partially supported by the ANR project MACK}

\maketitle

\section*{\footnotesize{\it{Dedicated to D.H.Phong on the occasion of his 60th birthday}}}

\text{ }

\section{Introduction}

Let $X$ be a compact K\"ahler manifold and $\a \in H^{1,1}(X,\R)$ be a big cohomology class. Recall that a cohomology class $\a \in H^{1,1}(X,\R)$ is {\it big} if it contains a {\it K\"ahler current},
i.e. a positive closed current  which dominates a K\"ahler form. 

Fix $\theta$ a smooth closed real $(1,1)$-form representing $\a$. We denote by  $PSH(X,\theta)$ 
 the set of all $\theta$-plurisubharmonic functions, 
i.e. those functions $\f:X \rightarrow \R \cup \{-\infty\}$ which can be written 
locally as the sum of a smooth and a plurisubharmonic function and such that the
current $\theta+dd^c \f$ is a closed positive current, i.e.: $\theta+dd^c \f\ge 0$ in the sense of currents. 
It follows from the $\partial\overline{\partial}$-lemma that any closed  positive
current $T$ in $\a$ can be written as $T=\theta+dd^c \f$ for some $\f \in PSH(X,\theta)$. 
We use the standard normalization
$$
d=\partial+\overline{\partial},
\hskip.5cm
d^c:=\frac{1}{2i \pi} (\partial-\overline{\partial}) 
\text{ so that } dd^c =\frac{i}{\pi} \partial \overline{\partial}.
$$

In general K\"ahler currents are too singular, so one usually
prefers to work with
 positive currents in $\a$ having {\it minimal singularities}. 
A positive current  $T=\theta+dd^c \f \in \a$ (resp. a $\theta$-plurisubharmonic function $\f$)  has minimal singularities 
if for every other positive current $S=\theta+dd^c \p \in \a$,
there exists $C \in \R$ such that $\p \leq \f+C$ on $X$. The function
$$
V_{\theta}:=\sup \{  v \, | \, v \in PSH(X,\theta) \text{ and } \sup_X v\le 0 \}
$$
is an example of $\theta$-psh function with minimal singularities. It satisfies $\sup_X V_{\theta}=0$. We let
$$
P(\a):=\{ x \in X \, | \, V_{\theta}(x) =-\infty \}
\text{ and }
NB(\a):=\{ x \in X\, | \, V_{\theta} \notin L_{loc}^{\infty}(\{x\}) \}
$$
denote respectively the {\it polar locus} and the {\it non bounded locus} of $\a$. 
The definitions clearly do not depend on the
choice of $\theta$ and coincide with the polar (resp. non bounded locus) of any $\theta$-psh function with minimal singularities.

\medskip

The purpose of this note is to show that if a big cohomology class contains {\it one} current having minimal singularities and exponentially continuous potentials, then there is actually plenty of  such currents:

\medskip
\noindent {\bf THEOREM.}
{\it 
Let $X$ be a compact K\"ahler manifold and let $\a \in H^{1,1}(X,\R)$ be a big cohomology class
such  that the polar locus $P(\a)$ coincides with  the non-bounded locus $NB(\a)$.

Fix $\theta \in \a$ a smooth representative and $T=\theta+dd^c \f$ a positive current in $\a$, where
$\f \in PSH(X,\theta)$.
Then there exists $\f_j \in PSH(X,\theta)$ a sequence of exponentially continuous 
$\theta$-plurisubharmonic functions
which have minimal singularities and decrease towards $\f$. 
}
\medskip

We say here that a $\theta$-psh function is {\it exponentially continuous} iff $e^{\f}:X \rightarrow \R$ is continuous. 
Observe that if there exists one exponentially continuous $\theta$-psh function with minimal singularities, then $P(\a)=NB(\a)$.

The technical condition $P(\a)=NB(\a)$ is thus necessary.
It is obviously satisfied when $\a$ is semi-positive, or even {\it bounded} (i.e. there exists
a positive closed current in $\a$ with bounded potentials, a condition that has become important in complex dynamics
recently, see \cite{DG}), since $P(\a)=NB(\a)=\emptyset$ in this case.
A more subtle example of a big and nef class $\a$ with $P(\a)=NB(\a) \neq \emptyset$ has been given in \cite[Example 5.4]{BEGZ}.

It is easy to construct $\theta$-plurisubharmonic functions $\p$  with $P(\p)  \subsetneq NB(\p)$, 
however we do not know of
a single example of a big class $\a$ for which $P(\a)$ is strictly smaller than $NB(\a)$.

\smallskip

Despite the relative modesty of its conclusion, this result relies on three important tools:

-the regularization techniques of Demailly as used in  \cite{BD},

-the resolution of degenerate complex Monge-Amp\`ere equations in big cohomology classes, as developed in \cite{BEGZ},

-and the viscosity approach to complex Monge-Amp\`ere equations \cite{EGZ2}.
\smallskip

The latter was developed in the case where $\a$ is both big and semi-positive, hence we need here to extend this technique
so as to cover the general setting of big classes. We stress that the global viscosity comparison principle
holds for a big cohomology class $\a$ if and only if $P(\a)=NB(\a)$ (Theorem \ref{thm:comparison}).

\smallskip
 
Our approximation result is new even when $\a$ is both big and semi-positive. Let us stress that
continuous $\theta$-plurisubharmonic functions are easy to regularize by using Richberg's technique \cite{R}.
As a consequence we obtain the following:

\medskip
\noindent {\bf COROLLARY.}
{\it 
 Let $(V,\omega_V)$ be a compact normal K\"ahler space and let $\f$ be a $\omega_V$-plurisubharmonic function on $V$.
Then there exists a  sequence $(\f_j)$ of smooth $\omega_V$-plurisubharmonic functions 
which decrease towards $\f$.
}
\medskip

When $\omega_V$ is a Hodge form, this regularization can be seen as a consequence of the extension result of \cite{CGZ}.

Approximation from above by regular objects is of central use in the theory of complex Monge-Amp\`ere operators,
as the latter are continuous along (and even defined through) such monotone sequences \cite{BT82}, while they are not
continuous with respect to the weaker $L^1$-topology \cite{Ceg}.

\medskip

\noindent {\it Plan of the note.} We first establish our main result when the underlying cohomology class is also
semi-positive (section \ref{sec:semipositive}), as the viscosity technology is already available \cite{EGZ2}; the corollary
follows then easily by using Richberg's regularization result. 
 We then (section \ref{sec:bigviscosity}) adapt the techniques of \cite{EGZ2} to the general context of big cohomology classes.
 The technical condition $P(\a)=NB(\a)$ naturally shows up as it is necessary for the global viscosity comparison principle to hold.
 We finally (section \ref{sec:bigpluripotential}) use recent stability estimates for big cohomology classes \cite{GZ3} to
 obtain continuous solutions of slightly more general Monge-Amp\`ere equations, which allow us to prove our main result.

\medskip

\noindent {\bf D\'edicace.} {\it  C'est un plaisir  de contribuer \`a  ce volume en l'honneur de Duong Hong Phong, dont nous appr\'ecions la g\'en\'erosit\'e, la vision et le bon go\^ut, tant math\'ematique que gastronomique !}

\section{The case of semi-positive classes} \label{sec:semipositive}

We fix once and for all $(X,\omega_X)$ a compact K\"ahler manifold of complex dimension $n$, 
$\a \in H^{1,1}(X,\R)$ a big cohomology class and $\theta$ a smooth closed $(1,1)$ form representing $\a$.

\subsection{Minimal vs analytic singularities}

Recall that  $\a$ is {\it semi-positive} if $\theta$ can be chosen to be a semi-positive form. 
In this case a $\theta$-psh function
has minimal singularities if and only if its is {\it bounded}. 
The easiest example of $\theta$-psh functions with minimal singularities are 
constant functions which are indeed $\theta$-psh iff $\theta$ is semipositive.

For more general $\a$,  $\theta$-psh function with minimal singularities
can be constructed as enveloppes, e.g.:
$$V_{\theta}:=\sup \{ v \, | \, v \in PSH(X,\theta) \text{ and } \sup_X v\le 0 \}.$$

Note that if $V \in PSH(X,\theta)$ is another function with minimal singularities, then
$V-V_{\theta}$ is globally bounded on $X$. Also if $\theta'=\theta+dd^c \rho$ is another smooth
form representing $\a$, then $PSH(X,\theta')=PSH(X,\theta)-\rho$ where
$\rho \in {\mathcal C}^{\infty}(X,\R)$ hence $V_{\theta}-V_{\theta'}$ is also globally bounded on $X$.

\begin{defi} 
The polar locus of $\a$ is
$$
P(\a):=\{ x \in X \, | \, V_{\theta}(x) =-\infty \}.
$$

The non-bounded locus of $\a$ is
$$
NB(\a):=\{ x \, | \, V_{\theta} \notin L_{loc}^{\infty}(\{x\}) \}
$$
\end{defi}

The observations above show that these definitions only depend on $\a$.
Clearly $P(\a) \subset NB(\a)$ and $NB(\a)$ is closed. We shall assume in the sequel 
that $P(\a)=NB(\a)$ which, as will turn out, is equivalent to saying that there
exists one exponentially continuous $\theta$-psh function with minimal singularities.

\smallskip

By definition $\a$ is big if it contains a {\it K\"ahler current}, i.e.
there is a (singular) positive current $T \in \a$ and $\e>0$ such that $T \geq \e \omega_X$.
It follows from the regularization techniques of Demailly (see \cite{Dem})
that one can further assume that $T$ has {\it analytic singularities}:

\begin{defi}
A positive closed current $T$ has analytic singularities if it can be locally written $T=dd^c u$, with
$$
u=\frac{c}{2}\log \left[ \sum_{j=1}^s |f_j|^2 \right]+v,
$$
where $c>0$, $v$ is smooth and the $f_j$'s are holomorphic functions.
\end{defi}

We let $\rm{Amp}(\a)$ denote the {\it ample locus} of $\a$, i.e. the Zariski open subset of all points
$x \in X$ for which there exists a K\"ahler current in $\a$ with analytic singularities which is smooth in
a neighborhood of $x$.

It follows from the work of Boucksom \cite{Bou} that one can find a single K\"ahler current $T_0=\theta+dd^c \p_0$
with analytic singularities in $\a$ such that 
$$
\rm{Amp}(\a)=X \setminus \rm{Sing} \, T_0.
$$
In the sequel we fix such a K\"ahler current $T_0$ and assume for simplicity that
$$ 
T_0 \geq \omega_X.
$$ 

Observe that  $\p_0$ is  exponentially continuous, however $\p_0$ does not have minimal singularities unless
$\a$ is K\"ahler (see \cite{Bou}).

\medskip

\noindent{\it Bounded vs continuous approximations.}
Fix $\f \in PSH(X,\theta)$ a $\theta$-psh function. 
It is easy to approximate $\f$ from above by a decreasing sequence of $\theta$-psh functions with {\it minimal singularities}.
Indeed we can  set
$$
\f_j:=\max(\f, V_{\theta}-j ) \in PSH(X,\theta).
$$
The latter have minimal singularities and decrease to $\f$ as $j \nearrow +\infty$.
This construction needs however to be refined to get exponentially continuous 
$\theta$-psh approximations with minimal singularities.  The mere existence of exponentially 
continuous $\theta$-psh functions $\p$ with minimal singularities is actually not obvious.

\subsection{Continuous approximations in the semi-positive case}

We show here our main result in the simpler case when $\a$ is both big and semi-positive.

\begin{thm} \label{thm:semipos}
Assume $\a \in H^{1,1}(X,\R)$ is big and semi-positive and let $\f \in PSH(X,\theta)$ be 
a $\theta$-plurisubharmonic function. 

Then there exists a sequence
of continuous $\theta$-plurisubharmonic functions which decrease towards $\f$.
\end{thm}

\begin{proof}
Fix $h_j$ a sequence of smooth functions decreasing to $\f$ (recall that $\f$ is upper semi-continuous)
and set
$$
\f_j=P(h_j):=\sup\{ u \, | \, u \in PSH(X,\theta) \text{ and } u \leq h_j \}.
$$
Observe that $\f_j \in PSH(X,\theta)$ and $\f_j \leq h_j$ hence
$$
\f \leq \f_{j+1} \leq \f_j.
$$
We claim that $\f=\lim \searrow \f_j$. Indeed set $\p:=\lim \searrow \f_j \geq \f$. Then
$\p \leq h_j$ for all $j$ and  $\p \in PSH(X,\theta)$,
hence $\p \leq \f$, so that $\p=\f$ as claimed.

It thus suffices to check that $\f_j$ is continuous. It follows from the work of Berman and Demailly \cite{BD}
that $\f_j$ has locally bounded Laplacian on the ample locus ${\rm Amp}({\alpha})$ of $\a$, with
$$
(\theta+dd^c \f_j)^n={\bf 1}_{\{P(h_j)=h_j\}} (\theta+dd^c h_j)^n
\text{ in } {\rm Amp}({\a}).
$$
The measure on the right hand side is absolutely continuous with respect to Lebesgue measure, with 
bounded density. It therefore follows from \cite[Theorem C]{EGZ2} that
$\f_j=P(h_j)$ is continuous.
\end{proof}

\begin{rem}
Observe that a key point in the proof above is that if $h$ is a smooth function on $X$, then
its $\theta$-plurisubharmonic projection $P(h)$ is a {\it continuous} $\theta$-plurisubharmonic function
with minimal singularities.

Although the proof is quite short in appearance, it uses several important tools: Demailly's regularization technique
(which is heavily used in \cite{BD}), and the viscosity approach for degenerate complex Monge-Amp\`ere equations developed in \cite{EGZ2}.

The proof of our main theorem follows exactly the same lines: the result of Berman-Demailly applies for general big
classes, while  \cite{BEGZ} produces solutions of complex Monge-Amp\`ere equations 
with minimal singularities in big cohomology classes. It thus remains to extend the viscosity
approach of \cite{EGZ2} to the setting of big cohomology classes, which is the contents of the next section.
\end{rem}

Since Richberg's regularization technique \cite{R} applies in a singular setting, we obtain the following
interesting consequence:

\begin{cor}
Let $(V,\omega_V)$ be a compact normal K\"ahler space and let $\f$ be a $\omega_V$-plurisubharmonic function on $V$.
Then there exists a  sequence $(\f_j)$ of smooth $\omega_V$-plurisubharmonic functions 
which decrease towards $\f$.
\end{cor}

\begin{proof}
Fix $\f \in PSH(X,\omega_V)$.
We can assume without loss of generality that $\f<0$ on $V$.
Let $\pi:X \rightarrow V$ be a desingularization of $V$ and set $\theta:=\pi^* \omega_V$. Then $\p:=\f \circ \pi \in PSH(X,\theta)$.

Since $\pi^*\{\omega_V\} \in H^{1,1}(X,\R)$ is big, it follows from our previous result that we can find continuous
$\theta$-psh functions $\p_j<0$ which decrease towards $\p$ on $X$. Since $\pi$ has connected fibers,
one easily checks that
$$
PSH(X,\theta)=\pi^*PSH(V,\omega_V),
$$
in particular there exists $\f_j \in PSH(V,\omega_V) \cap {\mathcal C}^0(V)$ such that $\p_j:=\f_j \circ \pi$,
with $\f_j$ decreasing to $\f$.

We can now invoke Richberg's regularization result
 \cite{R} (see also \cite{Dem}): using local convolutions
and patching, one can find smooth functions $(\f_{j,k})$ on $V$ which decrease to $\f_j$ as $k \rightarrow +\infty$
and such that $\f_{j,k} \in PSH(V, (1+\e_k) \omega_V)$ with $\e_k \searrow 0$. We can also assume that $\f_{j,k}<0$ on $V$.
Set finally
$$
u_j:=\frac{1}{(1+\e_j)} \f_{j,j} \in PSH(V,\omega_V) \cap {\mathcal C}^{\infty}(V).
$$
We let the reader check that $(u_j)$ still decreases to $\f$.
\end{proof}

\begin{rem}
When $\omega_V$ has integer class, i.e. when it represents the first Chern class of an ample line bundle on $V$,
the above result was obtained in \cite{CGZ} as a consequence of an extension result of
$\omega_V$-psh functions.
\end{rem}

\section{Viscosity approach in a big setting} \label{sec:bigviscosity}

 We set here the basic frame for the viscosity approach to the equation
$$
(\theta+dd^c \f)^n=e^{\e \f} v \leqno{(DMA_v^{\e})}
$$
where 
$v$ is a  volume form with nonnegative {\it continuous} density and $\e >0$ is  a real parameter.

\subsection{Viscosity sub/super-solutions for big cohomology classes}

To fit in with the viscosity point of view, we rewrite the Monge-Amp\`ere equation as
$$
e^{\e \f}v-(\theta+ dd^c \f)^n=0 \leqno{(DMA_v^{\e})}
$$

Let $x\in X$.
If $\kappa \in \Lambda^{1,1} T_{x} X$ we define $\kappa_+^n$ to be $\kappa^n$ if $\kappa \ge 0$ and $0$ otherwise. 
For a technical reason, we will also consider a slight variant of $(DMA_v^{\e})$,
$$
e^{\e \f}v-(\theta+ dd^c \f)_+^n=0 \leqno{(DMA_v^{\e})_+}
$$
If $\f^{(2)}_x$ is the $2$-jet at $x\in X$ of a ${\mathcal C}^2$ real valued function $\f$
we set
$$
F(\f^{(2)}_x)=F^{\e}_{v}(\f_x)=
\left\{
\begin{array}{ll}
e^{\e \f(x)} v_x-(\theta_x+ dd^c \f_x)^n & \text{ if } \theta+dd^c\f_x\ge 0 \\
+\infty & \text{otherwise}.
\end{array}
\right.
 $$
Recall the following definition from \cite{CIL}, \cite[Definition 2.3]{EGZ2}:

\begin{defi}
 A subsolution of $(DMA^{\e}_v)$ is an upper semi-continuous function $\f: X\to \R \cup \{-\infty\}$ such that
$\f\not \equiv -\infty$
and the following property is satisfied:
if $x_0\in X$ and  $q \in {\mathcal C}^2$, defined in a neighborhood of $x_0$, is such that $\f(x_0)=q(x_0)$ and
$$
 \f-q \  \text{ has a local maximum at} \ x_0,
 $$ 
 then $F(q^{(2)}_{x_0})\le 0$. 
 
 We say that $\f$ has minimal singularities if there exists $C>0$ such that
 $$
 V_{\theta}-C \leq \f \leq V_{\theta}+C
 \text{ on } X.
 $$
 \end{defi}

It has been shown in \cite[Corollary 2.6]{EGZ2} that a viscosity subsolution $\f$ of $(DMA^{\e}_v)$ is 
a $\theta$-psh function which satisfies $(\theta+dd^c \f)^n \geq e^{\e\f} v$ in the pluripotential sense of
\cite{BT82,BEGZ}.

\smallskip

We now slightly extend the concept of supersolution, so as to allow a supersolution to take $-\infty$ values:

\begin{defi}\label{supersol}
A supersolution of $(DMA^{\e}_v)$ is a supersolution of $(DMA^{\e}_v)_+$, that is, a   function 
$\f: X\to \R \cup \{\pm\infty\}$ such that 
$e^{\f}$ is lower semicontinuous, $\f\not \equiv +\infty$, $\f \not \equiv -\infty$
and the following property is satisfied:
if $x_0\in X$ and  $q \in {\mathcal C}^2$, defined in a neighborhood of $x_0$, is such that $\f(x_0)=q(x_0)$ and
$$
 \f-q \  \text{ has a local minimum at} \ x_0,
 $$
 then $F_+(q^{(2)}_{x_0})\geq 0$. 
 
 We say that $\f$ has minimal singularities if there exists $C>0$ such that
 $$
 (V_{\theta})_*-C \leq \f \leq (V_{\theta})_*+C
 \text{ on } X.
 $$
\end{defi}

Here $(V_{\theta})_*$ denotes the lower semi-continuous regularization of $V_{\theta}$.
 It is important to allow $-\infty$ values since we are trying to build a $\theta$-psh viscosity solution of $(DMA^{\e}_v)$: in
 general such a function will be infinite at polar points $x_0 \in P(\a)$. Note that we don't impose any condition at such points.

\begin{defi}
A viscosity solution of $(DMA^{\e}_v)$ is a function that is both a sub-and a supersolution. In particular a viscosity solution
$\f$ is automatically an exponentially continuous $\theta$-plurisubharmonic function.
\end{defi}

By comparison, a pluripotential solution of $(DMA^{\e}_v)$ is an
 usc function $\f \in L_{loc}^{\infty} (\rm{Amp}(\a)) \cap PSH(X,\omega)$
such that 
$$
(\theta+dd^c\f)^n_{BT}= e^{\e\f} v
\text{ in } \rm{Amp}(\a)
$$
in the sense of Bedford-Taylor \cite{BT82} (see \cite{BEGZ} for the slightly more general notion of non-pluripolar products): it follows from
\cite{BEGZ} that such a pluripotential solution automatically has minimal singularities, however there is no continuity information, especially
at points in $X \setminus \rm{Amp}(\a)$, as this set is pluripolar hence invisible from the pluripotential point of view.

\subsection{The big viscosity comparison principle}

 \begin{thm} \label{thm:comparison}
 Let $\a \in H^{1,1}(X,\R)$ be a big cohomology class and assume $\e>0$ and $v>0$. 
 Let $\f$ (resp. $\p$) be a subsolution (resp. supersolution) of $(DMA^{\e}_v)$ with minimal singularities, then
 $$
 \f \leq \p \text{ in } \rm{Amp}(\a).
 $$
 
 Moreover  $\f \leq \p$ on $X$ if and only if $P(\a)=NB(\a)$.
 \end{thm}
 
 \begin{proof}
 We can assume $\e=1$ without loss of generality. 
 
 
 



We let
$x_0 \in X$ denote a point that realizes the maximum of $e^\f-e^\p$ on $X$. If $x_0 \in P(\a)$, then we 
conclude trivially:  $\f(x_0)=-\infty$, hence $\max_X ( e^{\f}-e^{\psi} )\le 0$. 

Assume now $x_0 \not\in \rm{NB}(\a)$. Then $\f$ and $\p$ are locally bounded 
near $x_0$. 
Since $NB(\alpha)$ is closed, we can choose complex coordinates $(z^1, \ldots, z^n)$ near $ x_0$
defining a biholomorphism identifying an open neighborhood of ${x}_0$ in $X-\rm{NB}(\a)$ to the complex ball
 $B(0,5)\subset \C^n$ of radius $5$ sending ${x}_0$ to the origin. 

  We define $h \in {\mathcal C}^2(\overline{B(0,5)}, \R)$ to be a local potential 
smooth up to the boundary  for $\theta$ and extend it smoothly to $X$.
In particular $dd^c h = \theta$ and $w_- := \f + h$ is a bounded viscosity subsolution
of the equation
$$
(dd^c w)^n =e^{w} W
\text{ in } B(0,5)
$$ 
with $W$ a positive and continuous volume form. 
On the other hand $w^+= \p +h$ is a bounded viscosity supersolution of the same equation.
 
Now choose $C>0$ such that $\sup_{x\in \overline{B(0,4)} } \max(|\f(x)|,|\p(x)|) \le C/1000$ and
$\sup_{x\in \overline{B(0,4)} } |h(x)|\le C/10$. 
With this  constant $C>0$, construct as in \cite[p. 1076]{EGZ2}
a  smooth auxiliary function    $\varphi_3$ on $\overline{B(0,4)}^2$. 
Using the same notations as in \cite[p. 1077]{EGZ2}, 
fix $\beta > 0$ and consider $(x_{\beta},y_{\beta}) \in \overline{B(0,4)}^2$ such that:
\begin{eqnarray*}
M_{\beta}&=&\sup_{(x,y)\in \overline{B(0,4)}^2} w_-(x)-w^+(y) -\varphi_3(x,y)-\frac{\beta}{2} d^2(x,y)\\
&=&w_-(x_{\beta})-w^+(y_{\beta}) -\varphi_3(x_{\beta},y_{\beta})-\frac{\beta}{2} d^2(x_{\beta},y_{\beta}).
\end{eqnarray*}
By construction,  $\varphi_3$  is big enough outside $B (0,2)^2$ to ensure that
the sup is achieved at some point  $(x_{\beta}, y_{\beta}) \in B(0,2)^2$. 
Limit points $(x,y)$ of $(x_{\beta}, y_{\beta})$ satisfy $x=y$
and the construction forces $\varphi_3$ to vanishes to high order at such a limit point. 
Then, the argument of \cite[p. 1077-1078]{EGZ2}
based on Ishii's version of the maximum principle (see \cite{CIL}) applies verbatim
to prove that 
$$
\displaystyle\limsup _{\beta \to 0}w^+(x_{\beta})-w^-(y_{\beta})\ge 0
$$ 
and enables us to conclude that $\f\le \psi$.

However, if $x_0 \in NB(\a) \setminus P(\a)$, $\f (x_0)>-\infty$ since $\phi$ has minimal
singularities, while $\p(x_0)=-\infty$
since $\p$ has minimal singularities in the sense of definition \ref{supersol}
and $x_0\in  NB(\a)$ implies that $(V_{\theta})_*(x_0)=-\infty$.
 The global comparison principle thus fails  if  $P(\a)\not =NB(\a)$.

Let us now justify that in general we do have $\f\le \p$ on ${\rm Amp}(\a)$. 
 Let $T_0=\theta+dd^c \p_0$ be a K\"ahler current such 
that $\rm{Amp}(\a)=X \setminus Sing(T_0)$, $\p_0\le \p$ and $\p_0 \le 0$.
 Fix $\delta>0$ and consider $\f_{\delta}:=(1-\delta) \f+\delta \p_0+n\log(1-\delta)$. We claim that $\f_{\delta}$ is again a subsolution 
 of  $(DMA^{\e}_v)$. Indeed there is nothing to test on $Sing(T_0)$, while in ${\rm Amp}(\a)$
 $$
 (\theta+dd^c \f_{\delta})^n \geq (1-\delta)^n (\theta+dd^c \f)^n\ge (1-\delta)^n e^\f v \geq e^{\f_{\delta}} v,
 $$
 as follows easily by interpreting these inequalities in the pluripotential sense (see \cite[Proposition 1.11]{EGZ2}).
 
 Let $x_{\delta}$ be a point where the upper semi-continuous 
function $e^{\f_{\delta}}-e^\p$  attains its maximum.
 If $x_{\delta} \in Sing(T_0)$, then $e^{\f_{\delta}(x_{\delta})}=0$ hence 
 $$
 e^{\f_{\delta}} \leq e^{\p} \Rightarrow \f_{\delta} \leq \p
 \text{ on } X.
 $$
 If $x_{\delta} \in \rm{Amp}(\a)$, then both $\f_{\delta}$ and $\p$ are locally bounded near $x_{\delta}$ 
and the argument above leads to the conclusion
 that $e^{\f_{\delta}(x_{\delta})} \leq e^{\p(x_{\delta})}$ 
 hence $\f_{\delta} \leq \p$ on $X$. 
 Letting $\delta$ decrease to zero, we infer that $\f \leq \p$ in $\rm{Amp}(\a)$.
 \end{proof}

\subsection{Continuous solutions of big Monge-Amp\`ere equations}

 \begin{thm} \label{thm:visc}
  Let $\a \in H^{1,1}(X,\R)$ be a big cohomology class and assume $\e>0$ and $v>0$ is a continuous positive density. 
Then there exists a unique pluripotential solution $\f$ of $(DMA^{\e}_v)$  on $X$, such that
\begin{enumerate}
\item $\f$ is a $\theta$-plurisubharmonic function with minimal singularities,
\item $\f$ is a viscosity solution in $\rm{Amp}(\a)$ hence continuous there,
\item Its lower semicontinuous regularisation   $\f_*$  is a viscosity supersolution. 
\end{enumerate} 

If $P(\a)=NB(\a)$ then $\f$ 
is a viscosity solution of $(DMA^{\e}_v)$ on $X$, hence $e^{\f}$ is continuous on $X$.
 \end{thm}

\begin{proof} We can always assume that $\e = 1$.
 Since the comparison principle holds on $\rm{Amp}(\a)$, we can use Perron's method by considering the upper enveloppe of subsolutions. 
 
It follows from \cite{BEGZ} that the equation $(DMA^{1}_v)$ has a pluripotential solution
$\phi_0$ which is a $\theta-$plurisubharmonic function on $X$ satisfying 
the equation $(\theta + dd^c \phi_0)^n = e^{\phi_0} v$ weakly on $X$. 
Since the right hand side is a bounded volume form, it follows 
from the big version of Kolodziej's uniform estimates that $\f_0$ has minimal singularities.

Moreover  from the definition of a subsolution in the big case and
 \cite[Corollary 2.6]{EGZ2}, it follows that $\f_0$ is a viscosity subsolution to the equation $(DMA^{1}_v)$. 

On the other hand, since by \cite{BD}, $V_\theta$ satisfies the equation 
$(\theta + dd^c V_{\theta})^n = {\bf 1}_{\{V_{\theta} = 0\}} \theta^ n$ 
in the pluripotential sense and the right hand side is a bounded volume form, it follows  
that for some constant $C >> 1$ the function $\phi_1 := V_\theta + C$
 satisfies the inequality $(\theta + dd^c \phi_1)^n \leq e^{\phi_1} v$ in pluripotential sense. It follows 
therefore from the proof of \cite[Lemma 4.7(1)]{EGZ2} that $\psi_1 := (\phi_1)_*$ is a (viscosity) supersolution to the equation  $(DMA^{1}_v)$. 

We can now consider the upper envelope of (viscosity) subsolutions,

$$
 \f := \sup \{ \p \ ; \ \p \ \text{viscosity subsolution}, \  \phi_0 \leq \psi \leq \psi_1\},
$$
which is a subsolution with minimal singularities to the equation  $(DMA^{1}_v)$.
Using the bump construction (\cite{CIL}, \cite{EGZ2}) we can show that the lower semi-continuous regularization $\f_*$ of $\f$ is a (viscosity) supersolution with minimal singularities to the equation  $(DMA^{1}_v)$.

Therefore since the comparison principle holds on $\rm{Amp}(\a)$, 
it follows that $\f \leq \f_*$ on $\rm{Amp}(\a)$, hence $\f = \f_*$ on $\rm{Amp}(\a)$ is a viscosity solution to the equation  $(DMA^{1}_v)$ on $\rm{Amp}(\a)$.

If moreover $P (\a) = NB (\a)$ then we conclude by Theorem \ref{thm:comparison} that $\f = \f_*$ on $X$ is a viscosity solution to the equation  $(DMA^{1}_v)$ on $X$, hence $e^\f$ is continuous on $X$.

\end{proof}

\section{Pluripotential tools} \label{sec:bigpluripotential}

\subsection{Stability inequalities for big classes}
 
The following result is the main stability inequality established in \cite{GZ3}:
 
\begin{thm}  \label{thm:stability}
Assume $(\theta+dd^c\f_{\mu})^n=f_{\mu} \omega_X^n, (\theta+dd^c \f_{\nu})^n=f_{\nu} \omega_X^n$, 
where the densities $0 \leq f_{\mu},f_{\nu}$ are in   $L^p(\omega_X^n)$ for some $p>1$ and $\f_{\mu},\f_{\nu} \in PSH(X,\theta)$
are normalized by $\sup_X \f_{\mu}=\sup_X \f_{\nu}=0$. Then
 $$
  \Vert \f_\mu - \f_\nu \Vert_{L^{\infty} (X)} \leq M_{\tau} \Vert f_{\mu}- f_{\nu} \Vert_{L^1 (X)}^{\tau},
 $$
where  $M_{\tau} > 0$   only depends on upper bounds for the $L^p$ norms of $f_{\mu},f_{\nu}$ and
$$
0<\tau<\frac{1}{2^n(nq+1)-1}, \; \; \;
\frac{1}{p}+\frac{1}{q}=1.
$$
\end{thm}

\begin{cor} \label{cor:epsilon}
  Let $\a \in H^{1,1}(X,\R)$ be a big cohomology class and assume $\e>0$ and $v \geq 0$ is a
probability measure
  with $L^p$-density with respect to Lebesgue measure, where $p>1$.
Then there exists a unique $\theta$-plurisubharmonic function $\f$ with minimal singularities 
which is a pluripotential solution of $(DMA^{\e}_v)$ on $X$ 
  
Moreover $\f$ is continuous on $\rm{Amp}(\a)$  and
   if $P(\a)=NB(\a)$ then  $e^{\f}$ is also continuous on $X$.
 \end{cor}

 \begin{proof}
The first part follows from \cite{BEGZ} and we get a unique pluripotential solution $\varphi$
with minimal singularities. 
 Let $f$ denote the density of $v=f\omega_X^n$. We can approximate 
 $f$ by  continuous and positive
 densities by using convolutions, locally $f_{\delta}:=f \star \chi_{\delta}+\delta$, $\delta>0$.
Theorem \ref{thm:visc} insures that there exists a unique
 $\f_{\delta} \in PSH(X,\theta)$ solution to
 $$
 (\theta+dd^c \f_{\delta})^n=e^{\e \f_{\delta}} f_{\delta} \omega_X^n
 $$
 which has minimal singularities and is continuous in ${\rm Amp}(\a)$. 
It follows moreover from  \cite{BEGZ} Theorem \ref{thm:stability} that
  the functions $\f_{\delta}-V_{\theta}$ are uniformly bounded as $\delta \rightarrow 0^+$.

The family $\{\f_{\delta} \}_ {\delta>0}$ is compact in the $L^1$ topology
hence we can extract a sequence $(\delta_k)_k$ such that $\f_{\delta_k}$
converges almost everywhere and $\sup_X \f_{\delta_k}$ converges.  
 We can now apply the stability inequality (Theorem \ref{thm:stability})
 to $\tilde{\f}_{\delta}:=\f_{\delta}-\sup_X \f_{\delta}$
 to check that the functions $\tilde{\f}_{\delta_k}$ form a Cauchy sequence,
 so that the functions ${\f}_{\delta_k}$ form a Cauchy sequence too.  
The uniform limit $\phi=\lim {\f}_{\delta_k}$ has minimal singularities and satisfies 
$
 (\theta+dd^c \phi)^n=e^{\e \phi} f \omega_X^n
 $  in the pluripotential sense, hence coincides
with $\varphi$ .  
  \end{proof}

\subsection{The flat setting}

\begin{thm}  \label{thm:flat}
 Let $\a \in H^{1,1}(X,\R)$ be a big cohomology class and assume $v \geq 0$ is a volume form
  with non-negative $L^p$-density with respect to Lebesgue measure  such that 
  $\int_X v=\rm{Vol}(\a)$ and $p>1$.
Then there exists a unique $\theta$-plurisubharmonic function $\f$ with minimal singularities 
which is a pluripotential solution of $(DMA^{0}_v)$ on $X$ and such that $\int_X \f dv=0$.
  
Moreover $\f$ is continuous on $\rm{Amp}(\a)$  and
   if $P(\a)=NB(\a)$ then  $e^{\f}$ is also continuous on $X$.
 \end{thm}
 
 \begin{proof}
Fix $\e>0$ and let $\f_{\e}$ be the unique solution of $(DMA_v^{\e})$ given by Corollary \ref{cor:epsilon}.
It follows from  \cite{BEGZ} that the functions $\f_{\e}-V_{\theta}$ are uniformly bounded, hence 
the argument of Corollary \ref{cor:epsilon} enables to extract a Cauchy sequence
whose  limit
$\p$ is a solution of  $(DMA^{0}_v)$ which satisfies
$$
\int \p \, dv=\lim_{\e \rightarrow 0} \int  \left[ \frac{e^{\e \f_{\e}}-1}{\e} \right] \, dv=\int \f dv=0,
$$
hence $\f=\p$ (by the uniqueness result  proved in \cite{BEGZ}) has all required properties.
 \end{proof}

\begin{cor}
The function $V_{\theta}$ is continuous if and only if $P(\a)=NB(\a)$.
\end{cor}

\begin{proof}
Observe that if $e^{V_{\theta}}$ is continuous then $P(\a)=NB(\a)$, as points in $NB(\a) \setminus P(\a)$ 
correspond to points where $V_{\theta}$ is finite but not locally finite.

Conversely it follows from the work of Berman-Demailly \cite{BD} that $V_{\theta}$ has locally bounded Laplacian in $\rm{Amp}(\a)$
and satisfies
$$
(\theta+dd^c V_{\theta})^n={\bold{1}}_{\{V_{\theta}=0\}} \theta^n
\text{ in } \rm{Amp}(\a).
$$
Since $\theta^n$ is smooth and the density ${\bold{1}}_{\{V_{\theta}=0\}} $ is bounded, it follows from previous theorem that
$V_{\theta}$ is continuous on $X$ if $P(\a)=NB(\a)$.
\end{proof}

 \subsection{Conclusion}
 
 The proof of the main theorem proceeds now exactly as in that of Theorem \ref{thm:semipos}: if $\f$ is a given
 $\theta$-psh function, we approximate it from above by a decreasing sequence of smooth functions $h_j$ and set
 $$
 \f_j:=P(h_j) \in PSH(X,\theta).
 $$
  These functions have minimal singularities, decrease to $\f$ and solve the complex Monge-Amp\`ere
  equation
  $$
  (\theta+dd^c \f_j)^n={\bold{1}}_{\{\f_j=h_j\}} (\theta+dd^c h_j)^n=f_j \omega_X^n,
  $$
  where $f_j$ is bounded.
  It follows therefore from Theorem \ref{thm:flat} that $\f_j$ is continuous if $P(\a)=NB(\a)$.
  
  \smallskip
 
 Let us conclude by mentionning that we don't know any example of a compact K\"ahler manifold $X$ and a
big cohomology class $\a \in H^{1,1}(X,\R)$ such that $P(\a)$ does not coincide with $NB(\a)$.

\end{document}